\documentclass[12pt]{amsart}
\usepackage{a4wide}
\usepackage{amssymb,bm,mathrsfs}
\usepackage{amsthm}
\usepackage{amsmath}
\usepackage{amscd}
\usepackage{verbatim}
\usepackage[all]{xy}

\usepackage{hyperref}

\numberwithin{equation}{section}

\theoremstyle{plain}
\newtheorem{theorem}{Theorem}[section]
\newtheorem{corollary}[theorem]{Corollary}
\newtheorem{lemma}[theorem]{Lemma}
\newtheorem{proposition}[theorem]{Proposition}

\theoremstyle{definition}

\theoremstyle{remark}
\newtheorem{remark}[theorem]{Remark}

\newcommand{\R}{\mathbb{R}}

\newcommand{\Q}{\mathbb{Q}}
\newcommand{\Z}{\mathbb{Z}}

\newcommand{\C}{\mathbb{C}}
\renewcommand{\H}{\mathbb{H}}

\newcommand{\kzxz}[4]{\left(\begin{smallmatrix} #1 & #2 \\ #3 & #4\end{smallmatrix}\right) }

\renewcommand{\Im}{\operatorname{Im}}
\renewcommand{\Re}{\operatorname{Re}}

\newcommand{\calO}{\mathcal{O}}

\newcommand{\calS}{\mathcal{S}}

\newcommand{\eps}{\varepsilon}

\newcommand{\Mp}{\operatorname{Mp}}

\newcommand{\Mat}{\operatorname{Mat}}

\newcommand{\dv}{\operatorname{div}}

\newcommand{\Sym}{\operatorname{Sym}}

\newcommand{\GL}{\operatorname{GL}}

\newcommand{\rT}{\hspace{.1em}\ensuremath{{}^{\mathrm{t}}}}

\begin{document}

\title[Correction to ``Formal Fourier-Jacobi series'']{Correction to ``Kudla's Modularity Conjecture and Formal Fourier-Jacobi Series''}

\author[Jan H.~Bruinier and Martin Raum]{Jan Hendrik Bruinier and Martin Raum}
\address{Technische Universit\"at Darmstadt, Fachbereich Mathematik, Schlossgartenstraße 7,
D--64289 Darmstadt, Germany}
\email{bruinier@mathematik.tu-darmstadt.de}

\address{
Chalmers tekniska högskola och Göteborgs Universitet, Institutionen för
Matematiska vetenskaper, SE--412 96 Göteborg, Sweden
}
\email{martin@raum-brothers.eu}

\subjclass[2020]{Primary 11F46; Secondary 14C25}

\thanks{
The first author is supported in part by  the
DFG Collaborative Research Centre TRR 326 ``Geometry and Arithmetic of Uniformized Structures'', project number 444845124.
The second author was partially supported by Vetenskapsrådet grant 2023-04217.}


\begin{abstract}
We correct an error in Lemma 4.4 and its application in Theorem 4.5
of \cite{BR}.
 \end{abstract}

\maketitle



\section{Required corrections}

Throughout we use the notation of \cite{BR}.
The statement of Lemma 4.4 of \cite{BR} contains an error.
We thank Haocheng Fan and Liang Xiao for pointing this out to us.
Here we state a corrected version of the lemma.

\begin{lemma}
\label{lem:coversplit-new}
Let $W\subset \C^N$ be a simply connected domain. Let $Q(\tau,X)\in \calO(W)[X]$ be a monic irreducible polynomial which is normal over the quotient field of $\calO(W)$, and denote its  discriminant by  $\Delta_Q\in \calO(W)$.
Let $V\subset W$ be a connected open subset 
that has non-trivial intersection with every irreducible component of the divisor $D=\dv(\Delta_Q)$.
If $f$ is a holomorphic function on $V$ satisfying $Q(\tau,f(\tau))=0$ on $V$, then $f$ has a holomorphic continuation to $W$.
\end{lemma}

The normality condition is missing in the statement of Lemma 4.4 of \cite{BR}. It is 
required in the proof to ensure that the automorphism group of the branched covering $\tilde W$ of $W$ defined by $Q$ acts transitively on the fibers.
The condition that $W$ be simply connected is required to conclude
the global splitting of $Q$ from the local splitting.
It is an interesting question whether the hypothesis of normality in the above lemma can be dropped or weakened. 

The proof of Theorem~4.5 of \cite{BR} requires a version of the above lemma without the normality assumption in the context of Siegel modular varieties. Since we currently do not have a proof for this, we give a variant of the proof of Theorem~4.5 which does not rely on the lemma. Instead it uses Corollary~\ref{cor:pointwise-conv} below, which guarantees the pointwise convergence of symmetric formal Fourier-Jacobi series of cogenus~$1$ at torsion points, following the argument of \cite[Section 6]{AIP} and partly generalizing it in the genus aspect; see in particular Proposition~6.8 and Theorem~7.4 of~\cite{AIP}. Here we restate the theorem and give the corrected proof.

\begin{theorem}
\label{thm:establish-holmorphicity-new}
Let $Q=\sum_{i=0}^d a_i X^i\in \operatorname{M}_\bullet^{(g)}[X]$ be a nonzero polynomial of degree $d$ with coefficients
$a_i\in \operatorname{M}_{k_0+(d-i)k}^{(g)}$, and let 
\begin{gather*}
f=\sum_{m\geq 0} \phi_m(\tau_1,z)\,
  q_2^m\in \operatorname{FM}^{(g)}_k
\end{gather*}
be a symmetric formal Fourier-Jacobi series of cogenus~$1$ such that $Q(f)=0$. Then $f$ converges locally uniformly on $\H_g$
and defines an element of  $\operatorname{M}_k^{(g)}$. Here $q_2=e^{2\pi i\tau_2}$.
\end{theorem}

We call a symmetric formal Fourier-Jacobi series $f$ as above \emph{cuspidal}, if $\phi_0=0$. In this case the symmetry condition implies that all coefficients $\phi_m$ are Jacobi cusp forms. 

\begin{proof}[Proof of Theorem~\ref{thm:establish-holmorphicity-new}]
First, we assume that $Q$ is monic 
and that $f$ is cuspidal.
By Corollary~\ref{cor:pointwise-conv} below,  $f$ converges pointwise absolutely for all $\tau= \kzxz{\tau_1}{z}{\rT z}{\tau_2}\in \H_g$ for which $(\tau_1,z)$ defines a torsion point. This is a dense subset of $\H_g$. Moreover, by Proposition~\ref{prop:locally-bounded} the sequence of partial sums $\sum_{m=1}^M \phi_m(\tau_1,z)\, q_2^m$ for $M\in \Z_{>0}$ is locally bounded on $\H_g$.
By the Lemma of Ascoli (see e.g.~\cite[Chapter 1.4, Ex.~4]{FG}) a sequence $(h_j)$ of holomophic functions  on a domain in $\C^n$ converges locally uniformly, if it converges pointwise on a dense subset and is locally bounded. 

This implies that
$f$ converges locally uniformly on all of $\H_g$ and defines a holomorphic function there.
Since $\Mp_{2g}(\Z)$ is generated by the
embedded Jacobi group $\widetilde{\Gamma}^{(g-1,1)}$ and the
embedded group $\GL_{g}(\Z)$, we find that $f\in \operatorname{M}_k^{(g)}$.

Now assume that $Q$ is not necessarily monic and $f$ not necessarily cuspidal.
We choose a non-zero cusp form $f_c\in \operatorname{M}_l^{(g)}$ of some positive weight $l$.
Then the polynomial $R(X):= a_d^{d-1}f_c^d Q((a_d f_c)^{-1} X)\in
\operatorname{M}_\bullet^{(g)}[X]$ is monic, and the cuspidal formal Fourier-Jacobi series $h:= a_d f_c  f$
satisfies $R(h)=0$. Replacing in the above argument $f$ by $h$ and $Q$ by $R$, we see that $h$
defines a holomorphic Siegel modular form on $\H_g$. Consequently, $f=h/(a_df_c)$ is a meromorphic Siegel modular form.
On the other hand, by \cite[Lemma 4.3]{BR},  there exists an open
neighborhood $U\subset X_g$ of the boundary divisor $\partial
Y_g\subset X_g$ on which  $f$ converges absolutely and locally uniformly. 
Hence, $f$ is holomorphic on the inverse image $V\subset\H_g$ of $U$ under the natural
map $\H_g \to X_g$. Now it follows from \cite[Proposition 4.1]{BR} that the 
polar divisor of $f$ must be trivial, and therefore $f$ is in fact
holomorphic on $\H_g$.
\end{proof}

\begin{remark}
In recent work~\cite{BBHJ} it is shown that Runge's Theorem, reproduced as Theorem~3.8 in~\cite{BR}, is incorrect. In \cite{BR} that result is only used in the proof of Lemma~3.9, which  is also an immediate consequence of an analytic estimate due to J.~Wang, as explained in \cite[Remark~3.10]{BR}. This makes the proof independent of Theorem~3.8.
\end{remark}

\subsubsection*{Acknowledgement}
We thank Cris Poor for his helpful comments on this note.


\section{Pointwise convergence at torsion points}

Let $N$ be a positive integer. We call a point in $(\tau_1,z)\in\H_{g-1}\times \C^{g-1}$ an \emph{$N$-torsion point}, if $z= \tau_1 \lambda +\mu$ with $\lambda,\mu \in \frac{1}{N}\Z^{g-1}$, that is, if $z$ is an $N$-torsion point of the abelian variety $\C^{g-1}/(\tau_1 \Z^{g-1} + \Z^{g-1})$.
The action of the Jacobi group $\widetilde \Gamma^{(g-1,1)}$ preserves the set of $N$-torsion points.
The union of all $N$-torsion points in $\H_{g-1}\times \C^{g-1}$ for $N\in \Z_{>0}$ 
is called the set of torsion points. It is well known that this is a dense subset. 

We denote by $\Gamma(N)\subset \Mp_{2(g-1)}(\Z)$ the principal congruence subgroup  of level $N$.
If $\Gamma \subset \Mp_{2(g-1)}(\Z)$ is a congruence subgroup, we write  $\mathrm{S}_k^{(g-1)}(\Gamma)$ for the space of Siegel cusp  forms of weight $k$ and genus $g-1$ for $\Gamma$.
Recall that if $h$ is any cusp form of weight $k$ and genus $g-1$ for $\Gamma$, 
then the \emph{Hecke bound} 
\begin{gather*}
\|h\|_\infty= \sup_{\tau_1\in \H_{g-1}} \left(|h(\tau_1)|(\det(\Im\tau_1))^{k/2}\right)
\end{gather*}
is finite. It defines a norm on the space $\mathrm{S}_k^{(g-1)}(\Gamma)$.

\begin{proposition}
\label{prop:HB}
Let  $f=\sum_{m} \phi_m(\tau_1,z)\, q_2^m\in\operatorname{FM}^{(g)}_k$, and assume that $f$ is cuspidal.
Fix  $\lambda,\mu \in \frac{1}{N}\Z^{g-1}$.
Then the functions 
\begin{gather*}
  \eta_m(\tau_1)
:=
  e(m\, \rT\lambda \tau_1\lambda)\,
  \phi_m(\tau_1, \tau_1 \lambda + \mu)
\end{gather*}
belong to $\mathrm{S}_k^{(g-1)}(\Gamma(N^2))$, and we have the bound
\begin{gather*}
  \|\eta_m\|_\infty \ll_{g,k,N}
  m^{k + \frac{g-1}{2}}
\quad\text{for all $m>0$}
\text{.}
\end{gather*}
\end{proposition}

\begin{proof}
The first assertion is a standard fact for Jacobi forms at torsion points, see e.g.~\cite[Theorem~1.5]{Ziegler}). For later use we sketch the argument. It is based on the equality
\begin{gather*}
  e(m \rT{\lambda} \mu )\,
  \eta_m(\tau_1)\, e(m \tau_2)
=
  \big(
  \phi_m(\tau_1, z)\, e(m \tau_2) \big|_k\, \gamma
  \big)_{z = 0}
\quad\text{with }
  \gamma
=
  \left(\begin{smallmatrix}
  1 & 0 & 0 & \mu \\ 
  \rT{\lambda} & 1 & \rT{\mu} & 0 \\
  && 1 & -\lambda \\
  && 0 & 1
  \end{smallmatrix}\right)
\text{.}
\end{gather*}
When viewing the principal congruence subgroup~$\Gamma(N^2) \subset \Mp_{2(g-1)}(\Z)$ as embedded into~$\Mp_{2g}(\Z)$, the inclusion
$\gamma \Gamma(N^2) \gamma^{-1}\subset\Mp_{2g}(\Z)$
implies that~$\eta_m$ belongs to the space $\mathrm{S}_k^{(g-1)}(\Gamma(N^2))$ as stated. Note that the level is independent of the index~$m$.

Let $ \operatorname{Pos}_{g-1}(\Q)$ denote the subset of positive definite matrices in $\Sym_{g-1}(\Q)$.
Given a congruence subgroup~$\Gamma$ of~$\Mp_{2(g-1)}(\Z)$, by~\cite[\S13, Theorem~3]{Ma} and multiplicative symmetrization, there is a positive constant~$b$ that only depends on the genus and the weight  such that the following map is injective:
It sends~$h \in \mathrm{S}^{(g-1)}_k(\Gamma)$ to the collection of its Fourier coefficients~$c(h; n)$ for~$n \in \operatorname{Pos}_{g-1}(\Q)$ with~$n_{ii} < b\, [\Gamma : \Mp_{2(g-1)}(\Z)]$ for all~$1 \le i \le g-1$. In particular, since the index of~$\Gamma(N^2)$ is bounded by~$N^{8 (g-1)^2}$, we obtain a norm on 
$\mathrm{S}_k^{(g-1)}(\Gamma(N^2))$ by
\begin{gather*}
  \| h \|_{\mathrm{FE}}
=
  \sum_{n \in \calS} |c(h;n)|
\quad\text{with }
  \calS
=
  \big\{
  n \in \operatorname{Pos}_{g-1}(\Q)
  \mathrel{:}\;
  2 N^2 n_{ij} \in \Z,\,
  n_{ii} < b N^{8 (g-1)^2}
  \big\}
\text{.}
\end{gather*}
Observe that~$\calS$ is a finite set, which depends on~$b$ and~$N$, but not on~$m$. 
Since~$\mathrm{S}^{(g-1)}_k(\Gamma(N^2))$ 
is finite dimensional, norm comparison shows that~$\|\eta_m\|_\infty \ll_{g,k,N} \| \eta_m \|_{\mathrm{FE}}$, where the implied constant is independent of~$m$. 
Hence, 
it suffices to bound individual Fourier coefficients of~$\eta_m$, that is, we have to show that
\begin{gather*}
  |c(\eta_m; n)| \ll_{g,k,N} m^{k + \frac{g-1}{2}}
  \quad\text{for all } n \in \calS
\end{gather*}
to prove the proposition.

For $y\in \Sym_g(\R)$ and $u\in \GL_g(\R)$ we write $y[u]= \rT u y u$.
Returning to the relation between~$\eta_m$ and~$\phi_m$ via the transformation~$\gamma$ in the beginning of the proof, we see that
\begin{gather*}
  |c(\eta_m; n)|
\le
  \sum_r
  \Big|
  c\big( f; t\big[ \left(\begin{smallmatrix} 1 & 0 \\ -\rT{\lambda} & 1\end{smallmatrix}\right) \big] \big)
  \Big|
\quad\text{with }
  t
=
  \left(\begin{smallmatrix}
  n & \frac{1}{2} r \\
  \frac{1}{2} \rT{r} & m
  \end{smallmatrix}\right)
\text{,}
\end{gather*}
where the sum runs over all~$r \in \frac{1}{2 N^2} \Z^{g-1}$ such that~$t$ is positive definite. In particular,
we can estimate their number by
\begin{gather*}
  \#\big\{
  r \in \tfrac{1}{2 N^2} \Z^{g-1}
  \mathrel{:}\;
  r_i^2 <   4 m b N^{8 (g-1)^2} \text{ for all~$1 \le i \le g-1$}
  \big\}
\ll_{g,N}
  m^{\frac{g-1}{2}}
\text{.}
\end{gather*} 
Therefore, the desired bound 
for~$|c(\eta_m; n)|$ follows once we establish the bound
\begin{gather*}
  \Big|
  c\big( f; t\big[ \left(\begin{smallmatrix} 1 & 0 \\ -\rT{\lambda} & 1 \end{smallmatrix}\right) \big] \big)
  \Big|
\ll_{g,k,N}
  m^k
\end{gather*}
for all $n \in \calS$, $r \in \tfrac{1}{2 N^2} \Z^{g-1}$, and $m\in \Z_{>0}$ such that 
$t$ as above 
is positive definite.

Let $M\in \Z_{>0}$ be the denominator of $\lambda$. Then $M$ divides $N$ and the pair $(M\lambda,M)$ defines a vector in $\Z^{g}$ with coprime entries. This implies that there exists~$u \in \GL_{g}(\Z)$ such that
\begin{gather*}
  \left(\begin{smallmatrix} 1 & 0 \\ -\rT{\lambda} & 1 \end{smallmatrix}\right)
  u
=
  \left(\begin{smallmatrix} \rho & \xi \\ 0 & M^{-1} \end{smallmatrix}\right) s
\quad\text{with }
  s
=
  \left(\begin{smallmatrix}
  0       & 0 & 1 \\
  0 &1_{g-2} & 0 \\
  1       & 0 & 0
  \end{smallmatrix}\right)
\text{,}
\end{gather*}
where~$\rho \in \Mat_{g-1}(\Z)\cap \GL_{g-1}(\Q)$ and $\xi \in \Z^{g-1}$.
We may choose~$u$ in such a way that~$n[\rho]$ is Minkowski reduced. 
Comparing determinants on both sides, we see that~$|\det(\rho)| = M$. The symmetry of~$f$ guarantees that
\begin{gather*}
\bigl|c\big( f; t\bigl[ \left(\begin{smallmatrix} 1 & 0 \\ -\rT{\lambda} & 1 \end{smallmatrix}\right) \bigr] \big)\bigr|
=
 \bigl| c\big( f; t\bigl[ \left(\begin{smallmatrix} 1 & 0 \\ -\rT{\lambda} & 1 \end{smallmatrix}\right) u \bigr] \big)\bigr|
=
 \bigl| c\big( f; t' \big)\bigr|
\quad\text{with }
  t'
=
  \left(\begin{smallmatrix}
  n' & \frac{1}{2} r' \\ \frac{1}{2} \rT r' & m'
  \end{smallmatrix}\right)
=
  t\bigl[ \left(\begin{smallmatrix} \rho & \xi \\ 0 & M^{-1} \end{smallmatrix}\right) s \bigr]
\text{.}
\end{gather*}

The bottom right entry~$m'$ of~$t'$ equals the top left entry $( n[\rho] )_{11}$ of the Minkowski reduced symmetric matrix~$n[\rho]$. The Hermite bound together with the bound 
for the diagonal entries of~$t$ allows us to estimate
\begin{gather*}
  \big( n[\rho] \big)_{11}
\ll_g
  \big( \det(\rho)^2 \det(n) \big)^{\frac{1}{g-1}}
\le
  \big( M^2\, b^{g-1} N^{8(g-1)^3} \big)^{\frac{1}{g-1}}
\ll_g
  N^{\frac{2}{g-1}}\, b N^{8(g-1)^2}
\ll_{g,k,N}
  1
\text{.}
\end{gather*}
In particular, we can estimate~$c(f; t')$ in terms of finitely many~$\phi_{m'}$. Using the Hecke bound for Fourier coefficients of the associated vector-valued Siegel modular forms, we obtain that
\begin{gather*}
  |c(f; t')|
=
  |c(\phi_{m'}; n', r')|
\ll_{g,k,N}
  \det(t')^k
=
  \det(t)^k
\le
  \big( \det(n) m \big)^k
\ll_{g,k,N}
  m^k
\text{,}
\end{gather*}
since $n$ is contained in the finite set $\calS$.
\end{proof}

\begin{corollary}
\label{cor:pointwise-conv}
Let  $f=\sum_{m} \phi_m\, q_2^m\in\operatorname{FM}^{(g)}_k$ be cuspidal.
Fix an $N$-torsion point $(\tau_1,z)\in\H_{g-1}\times \C^{g-1}$, and put $C=\rT(\Im z)(\Im \tau_1)^{-1} (\Im z)$. Then the series 
\begin{gather*}
\sum_{m} \phi_m(\tau_1,z) \, q_2^m \in \C[[q_2]]
\end{gather*}
converges absolutely on the disc $|q_2|<e^{-2\pi C}$ and defines a holomorphic function in $q_2$ there. In particular, $f$ converges pointwise absolutely for all $\tau= \kzxz{\tau_1}{z}{\rT z}{\tau_2}\in \H_g$ for which $(\tau_1,z)$ defines a torsion point. 
\end{corollary}

\begin{proof}
By assumption there exist  $\lambda,\mu \in \frac{1}{N}\Z^{g-1}$ such that $z= \tau_1 \lambda +\mu$. This implies $C=\rT\lambda(\Im \tau_1) \lambda$.
According to Proposition \ref{prop:HB}, there exist $A,B>0$ such that 
\begin{align*}
|\phi_m(\tau_1,z)| & = |e(-m\, \rT\lambda \tau_1\lambda) \eta_m(\tau_1)|
\leq  e^{2\pi m\, \rT\lambda \Im(\tau_1)\lambda} (\det\Im \tau_1)^{-k/2}  \cdot A\cdot m^{B}
\end{align*}
for all $m > 0$. 
Hence, for any $\eps>0$ and $|q_2|\leq e^{-2\pi (C+\eps)}$ we obtain
\begin{align*}
|\phi_m(\tau_1,z)\,  q_2^m| & \leq   |\phi_m(\tau_1,z)| e^{-2\pi m (C+\eps)}
\leq  A(\det \Im\tau_1)^{-k/2} m^{B}e^{-2\pi \eps m}.
\end{align*}
This implies the first statement. For the last one we note in addition that $\Im(\tau)$ is positive definite, if and only if $\Im(\tau_1)$ and $\Im(\tau_2) -\rT\Im(z)(\Im \tau_1)^{-1} \Im(z)$ are positive definite. 
\end{proof}

\begin{proposition}
\label{prop:locally-bounded}
Let $Q=\sum_{i=0}^d a_i X^i\in \operatorname{M}_\bullet^{(g)}[X]$ be a monic polynomial of degree $d$ with coefficients
$a_i\in \operatorname{M}^{(g)}_{(d-i)k}$.
Let  $f=\sum_{m} \phi_m\, q_2^m\in\operatorname{FM}^{(g)}_k$ be cuspidal, and assume $Q(f)=0$.
Then the sequence of partial sums
\begin{align}
\label{eq:ps}
\sum_{m=1}^M \phi_m(\tau_1,z) q_2^m,\quad M\in \Z_{>0},
\end{align}
is locally bounded on $\H_g$.
\end{proposition}

\begin{proof}
Let $\varrho:\H_g\to \R_{>0}$ be the surjective map taking 
$\tau =  \kzxz{\tau_1}{z}{\rT z}{\tau_2}$ to $\varrho(\tau)=  \Im(\tau_2) -\rT(\Im z)(\Im \tau_1)^{-1} (\Im z)$,
where $\tau_1\in \H_{g-1}$, $z\in \C^{g-1}$, and $\tau_2\in \H$. 
Let $U\subset \H_{g-1}\times \C^{g-1}$ be a compact subset.
For any small $\eps>0$ we define a compact subset of $\H_g$ by 
\begin{gather*}
K_\eps(U) = \{ \tau\in \H_g\mid \; (\tau_1,z)\in U, \; \varrho(\tau) \in [\eps, 1/\eps],\; \Re(\tau_2)\in [-1/\eps,1/\eps]\}.
\end{gather*}

Consider a torsion point $(\tau_1,z)\in U$ and $\tau_2\in \H$ such that the corresponding matrix $\tau$ as above is contained in $K_\eps(U)$. Specializing the polynomial equation $Q(f)=0$ to such $\tau$, by Corollary~\ref{cor:pointwise-conv} we get the relation 
\begin{align}
\label{eq:rel2}
f^d+a_{d-1} f^{d-1} + \dots + a_0 =0
\end{align}
of convergent powers series in $q_2$ on the disc $|q_2|< e^{-2\pi C}$. Here $C=\rT(\Im z)(\Im \tau_1)^{-1} (\Im z)$. It implies the bound 
\begin{align}
\label{eq:rel3}
|f|\leq \sup_{\tau\in K_\eps(U)}\bigg( 1+ \sum_{i=0}^{d-1}|a_i(\tau)|\bigg) =:D_\eps(U).
\end{align}
In fact, this is clear if $|f|\leq 1$, and it directly follows from \eqref{eq:rel2} if $|f|\geq 1$.
Since $K_\eps(U)$ is compact and the $a_i$ are continuous, the quantity $D_\eps(U)$ is finite. The point here is that \eqref{eq:rel3} holds uniformly for all points $\tau\in K_\eps(U)$ for which $(\tau_1,z)$ is a torsion point.

Still for such a $\tau$, we may compute $\phi_m$ by means of the Fourier integral 
\begin{gather*}
\phi_m(\tau_1,z) =  \int_0^1 f\kzxz{\tau_1}{z}{\rT z}{u_2+iv_2} e^{-2\pi i m(u
_2+iv_2)}\,du_2,
\end{gather*}
for any fixed $v_2>C$. Taking $v_2=C+\eps$ and using \eqref{eq:rel3} we get
$|\phi_m(\tau_1,z)|
\leq D_\eps(U) e^{2\pi m(C+\eps)}$.
Now, if $\tau$ is actually contained in  the subset $K_{2\eps}(U)\subset K_\eps(U)$, we may estimate
\begin{align*}
|\sum_{m=1}^M \phi_m(\tau_1,z) q_2^m| &\leq \sum_{m=1}^M |\phi_m(\tau_1,z)| e^{-2\pi m (C+2\eps)}
\leq D_\eps(U) \sum_{m=1}^M e^{-2\pi m \eps} \leq \frac{  D_\eps(U)e^{-2\pi  \eps}}{1-e^{-2\pi  \eps}}.
\end{align*}
This shows that the sequence of partial sums \eqref{eq:ps} is bounded on the subset of $\tau \in K_{2\eps}(U)$ for which $(\tau_1,z)$ is a torsion point. Since this is a dense subset of $K_{2\eps}(U)$ and since the partial sums are continuous, we may conclude that  \eqref{eq:ps} is bounded on the whole $K_{2\eps}(U)$. Since every point in $\H_g$ has a neighborhood of the form $K_{2\eps}(U)$ for suitable $U$ and $\eps$, we obtain the assertion.
\end{proof}


\end{document}